\documentclass{birkjour}

\title[On the IYB-property in some solvable groups]{On the IYB-property in some solvable\\ groups}
\author{Florian Eisele}
\date{} 

\usepackage[english]{babel}
\usepackage{url}
\usepackage{amssymb}
\usepackage{amsmath}
\usepackage{amsthm}
\usepackage{amsfonts}
\usepackage{graphicx}
\usepackage{bibgerm}

\newtheorem{defi}{Definition}[section]
\newtheorem{remark}[defi]{Remark}

\newtheorem{thm}[defi]{Theorem}

\newtheorem{prop}[defi]{Proposition}

\newcommand{\Z}{\mathbb Z}
\newcommand{\F}{\mathbb F}
\newcommand{\iso}{\cong}
\newcommand{\End}{\operatorname{End}}
\newcommand{\Aut}{\operatorname{Aut}}
\newcommand{\Sym}{\Sigma} 
\newcommand{\id}{{\operatorname{id}}}
\newcommand{\GL}{\operatorname{GL}}
\newcommand{\matZ}[1]{\left( \begin{array}{rr} #1 \end{array} \right)}
\newcommand{\matD}[1]{\left( \begin{array}{rrr} #1 \end{array} \right)}
\newcommand{\vect}[1]{\left( \begin{array}{r} #1 \end{array} \right)}
\newcommand{\Rad}{\operatorname{Rad}}
\newcommand{\Soc}{\operatorname{Soc}}

\begin{document}

\address{%
Vrije Universiteit Brussel\\
Vakgroep Wiskunde\\
Pleinlaan 2\\
B-1050 Brussels\\
Belgium}

\email{feisele@vub.ac.be}



\begin{abstract}
A finite group $G$ is called Involutive Yang-Baxter (IYB) if there exists a bijective $1$-cocycle 
$\chi: G \longrightarrow M$ for some $\mathbb Z G$-module $M$. It is known that every IYB-group is solvable, but it is still an open question whether the converse holds. A characterization of the IYB property by the existence of an ideal $I$ in the augmentation ideal $\omega \mathbb Z G$ complementing the set $1-G$ lead to some speculation that there might be a connection with the isomorphism problem for $\mathbb Z G$.
In this paper we show that if $N$ is a nilpotent group of class two and $H$ is an IYB-group of order coprime to that of $N$, then
$N\rtimes H$ is IYB. The class of groups that can be obtained in that way (and hence are IYB) contains in particular Hertweck's famous counterexample to the isomorphism conjecture as well as all of its subgroups. We then investigate what an IYB structure on Hertweck's counterexample looks like concretely. 
\end{abstract}

\maketitle

\section{Introduction}

Recently there has been considerable interest in the problem of characterizing those finite groups $G$ for which there exists a bijective $1$-cocycle
$\chi:\ G \longrightarrow M$ ($M$ being a finite $\Z G$-module with $|M|=|G|$).
A group for which such a cocycle exists is called an Involutive-Yang-Baxter group. It is easily seen that preimages of submodules of $M$ are subgroups of $G$,
and hence Hall subgroups of all possible orders exist in $G$ (since submodules of corresponding order exist in $M$). By a well-known theorem of Hall it thus follows that  
IYB-groups are solvable. This was first observed in \cite{EtingofIYB}, which is also the article that first introduced the notion of an IYB-group. Whether, conversely, every solvable group is IYB is an open question. There is an equivalent formulation of the IYB property, stating that a group is IYB if and only if there is a left ideal $I$ in the augmentation ideal $\omega \Z G$ such that the elements of the form $1-g$ form a complement of 
$I$, that is,  they form a set of residue class representatives for the quotient $\omega \Z G / I$. The relation with bijective $1$-cocycles is evident: given such an ideal $I$, the map $G \longrightarrow \omega \Z G / I$ which maps $g$ to $1-g+I$ is a bijective $1$-cocycle. Now if $I$ were a two-sided ideal, then 
$G$ would be the circle group of the radical ring $\omega \Z G / I$ and hence it would be determined by its group ring. But even if $I$ is not two-sided, the group $(1+I)\cap \mathcal U(\Z G)$ complements the trivial units $\pm G$ in the unit group  $\mathcal U(\Z G)$. However it is not clear what implications, if any, the existence of a (non-normal) complement to the trivial units has. In the introduction to \cite{JOCBraces} it is conjectured that there is a connection between the question whether every solvable group is IYB and the integral isomorphism problem. In this short article we show that the counterexample to the isomorphism problem given by Hertweck in \cite{HertweckIso} cannot serve as an example for a solvable group which is not IYB (and we also exclude a plethora of similarly constructed groups). 
This is done in some generality in Theorem \ref{thm_semidirect_nilpot}, which states that the semidirect product of a nilpotent group of class two with any IYB-group is again IYB. 
A deeper connection with the isomorphism problem therefore seems unlikely. We also give an explicit description of a bijective $1$-cocycle on Hertweck's counterexample to the isomorphism problem. This seems interesting as it shows that in the context of IYB-groups there is apparently nothing special about Hertweck's group, which reinforces the conjecture that all soluble groups might be IYB.

\section{Involutive Yang-Baxter structures and their construction}

\begin{defi}
	\begin{enumerate}
	\item A finite group $G$ is called \emph{Involutive Yang-Baxter (IYB)} if there is a (left) $\Z G$-module $M$ and 
	a bijective $1$-cocycle $\chi: G \longrightarrow M$. We call the pair $(M, \chi)$ an \emph{IYB-structure} on the group $G$.
	\item Assume we are given, in addition to $G$, a group $A$ which acts on $G$ (from the left) by automorphisms.
	If $M$ is a $G\rtimes A$-module and 
	$\chi:\ G \longrightarrow M|_G$ is a bijective $1$-cocycle with the property that
	\begin{equation}
		\chi({^a g}) = a \chi(g) \quad \textrm{ for all $a \in A$ } 
	\end{equation}
	then we call the pair $(M, \chi)$ an \emph{$A$-equivariant IYB-structure} on $G$. Note that a $1$-equivariant IYB-structure simply an ordinary IYB-structure.
	\item Let $(M, \chi)$ and $(M', \chi')$ be two $A$-equivariant IYB-structures on $G$. Then $(M, \chi)$ and $(M',\chi')$ are called isomorphic if 
	there is a $G\rtimes A$-module isomorphism $\varphi: M \longrightarrow M'$ such that $\chi'=\varphi \circ \chi$.
	\end{enumerate}
\end{defi}

The following two propositions show that a semidirect product  $N\rtimes H$ with factors of coprime orders is IYB if and only if 
$N$ has an $H$-equivariant IYB-structure and $H$ is IYB. For a $p$-group that has an automorphism group with a large $p'$-part there may be 
significantly fewer IYB-structures equivariant with respect to these automorphisms than there are IYB-structures in general. This is evidenced for instance by Theorem \ref{thm_upper_triang} below, as well as by abelian groups where the situation is similar. So finding a $p$-group which has no equivariant 
IYB-structure with respect to some solvable $p'$-group acting on it by automorphisms would lead to a solvable group which is not IYB. However, at this point, it is 
unclear whether such a $p$-group exists, and if so, where to look for it.
\begin{prop}[{see also \cite[Theorem 3.4]{CFJIYB}}]
	Let $G= N \rtimes H$ be a finite group. If $H$ is IYB and $N$ has an $H$-equivariant IYB-structure, then $G$ is IYB. 
\end{prop}
\begin{proof}
	Let $(M_H, \chi_H)$ be an IYB-structure on $H$, and let $(M_N, \chi_N)$ be an $H$-equivariant IYB-structure on $N$. 
	Note that $M_N$ is a $G$-module by definition, and $M_H$ can be construed as a $G$-module by letting $N$ act trivially. Then
	$M := M_N \oplus M_H$ is a $\Z G$-module and 
	$\chi: G \longrightarrow M: n\cdot h \mapsto (\chi_N(n),\chi_H(h)) \textrm{  (where $n\in N$ and $h\in H$)  }$ defines a bijective $1$-cocycle on $G$.
\end{proof}

The following proposition shows that the converse of the above is true if $N$ is a normal Hall-subgroup of $G$, that is, if $(|N|, |H|)=1$.
\begin{prop}
	Assume $G= N\rtimes H$ is IYB and assume that $N$ is a Hall-subgroup of $G$. Then $H$ and $N$ are IYB, and $N$ has an $H$-equivariant IYB-structure.
\end{prop}
\begin{proof}
	Let $(M,\chi)$ be an IYB-structure on $G$. Decompose $M = M_N \oplus M_H$, where $|M_N|=|N|$ and $|M_H|=|H|$ (clearly this is possible 
	since $M$ is an $RG$-module with $R=\Z / |G| \iso \Z / |N| \oplus \Z / |H|$). Since $G$ is IYB it must in particular be solvable, and therefore
	all Hall-subgroups of a given order are conjugate. Moreover the preimages of submodules of $M$ under $\chi$ form subgroups of $G$ (this is an elementary computation).  Since $N$ is a normal Hall-subgroup it follows that $N=\chi^{-1}(M_N)$. Moreover there is a $g\in G$ such that $\chi^{-1}(M_H)={^gH}$. The map
	$\tilde \chi := g^{-1}\cdot \chi({^g -})$ is also a bijective $1$-cocycle, and $\tilde \chi^{-1}(M_H)=H$. Clearly the restricted maps $\tilde\chi|_H: H \longrightarrow M_H$
	and $\tilde\chi|_N: N \longrightarrow M_N$ are bijective $1$-cocycles. All that is left to verify if that $\tilde\chi|_N$ is $H$-equivariant.
	But
	\begin{equation}
		\underbrace{\tilde{\chi} ({^hn})}_{\in M_N} = \underbrace{hn\tilde{\chi}(h^{-1})}_{\in M_H}+\underbrace{h\tilde{\chi}(n)}_{\in M_N} + \underbrace{\tilde{\chi}(h)}_{\in M_H}
	\end{equation}
	Since $M=M_N\oplus M_H$ is a direct sum it follows $\tilde{\chi} ({^hn})=h\tilde{\chi}(n)$. \textit{(As a side note: it also follows that $N$ acts trivially on $M_H$)}
\end{proof}

In order to prove Theorem \ref{thm_semidirect_nilpot} below we need a few facts on the augmentation ideal in an integral group ring. In what follows we denote the 
augmentation ideal in a group ring $R G$ by $\omega RG$ and its $i$-th power by $\omega^i RG$.

\begin{remark}[{see \cite[Lemma 9.3.6 and Corollary 9.3.7]{MiliesSehgal} and \cite[Theorem 7.1]{PassiDimensionSubgroups}}]
\label{rem_dim_subgroup}
	We are going to need the following two well-known facts. Let $G$ be a finite group.
	\begin{enumerate}
	\item There is an isomorphism of abelian groups
	\begin{equation}
		G/[G,G] \stackrel{\sim}{\longrightarrow} \omega \Z G / \omega^2 \Z G:\ g[G,G] \mapsto 1-g+\omega^2 \Z G
	\end{equation}
	\item $\omega^2 \Z G \cap (1-G) = 1-[G,G]$ and $\omega^3 \Z G \cap (1-G) = [[G,G],G]$ (the more general 
	assertion that $\omega^i \Z G \cap (1-G)$ is equal to the $i$-th term in the lower central series is, or rather was, known as the \emph{dimension subgroup conjecture}; it is wrong in general, but true for odd order groups). 
	\end{enumerate}
\end{remark}

\begin{thm}[{see \cite[Theorem 1.24]{SandlingSplitting}}]\label{thm_sandling}
	Let $G$ be a finite group. Then the embedding
	\begin{equation}
		[G,G]/[[G,G],G] \longrightarrow \omega^2 \Z G / \omega^3 \Z G:\ g \mapsto 1-g
	\end{equation}  
	splits (as a homomorphism of abelian groups).
\end{thm}

\begin{thm}\label{thm_semidirect_nilpot}
	Assume $N$ is a nilpotent group of class two with a group $H$ of coprime order acting on it by automorphisms. Then $N$ possesses an $H$-equivariant 
	IYB-structure. In particular, $N\rtimes H$ is IYB if and only if $H$ is IYB.
\end{thm}
\begin{proof}
	We can assume without loss that $N$ is a $p$-group for some prime $p$ which does not divide the order of $H$. Let $\Z_p$ denote the $p$-adic integers.
	Consider the $\Z_p H$-module $M := \omega^2 \Z_p N / \omega^3 \Z_pN$. 
	We claim that the set $1-N'+\omega^3 \Z_pN$ (that is, the elements $1-n + \omega^3 \Z_pN$ for $n \in N'=[N,N]$) form an $H$-submodule of $M$. Clearly $H$ maps elements of the form
	$1-n$ to elements of the same form, so all we need to check is that these elements form a $\Z_p$-submodule  (of course this was already implicitly used in the formulation of Theorem \ref{thm_sandling}). So assume $n, m \in N' \subseteq Z(N)$. We get
	\begin{equation}
		(1-n) + (1-m) = (1-mn) + (1-m)(1-n) \equiv 1-mn \mod \omega^3 \Z_p N 
	\end{equation} 
	Now we will show that $1-N'+\omega^3 \Z_pN$ is complemented in $M$ as a $\Z_pH$-module. Since every $\Z_pH$-module is relatively $1$-projective it follows that
	$1-N'+\omega^3 \Z_pN$ is complemented as a $\Z_pH$-module if and only if it is complemented as a $\Z_p$-module. A more elementary way of stating this is that if 
	$\pi \in \End_{\Z_p}(\omega^2 \Z_pN/\omega^3\Z_p N)$ is a projection to $1-N'+\omega^3 \Z_pN$, then
	\begin{equation}
		\widehat \pi:\ \omega^2 \Z_pN/\omega^3\Z_p N \longrightarrow \omega^2 \Z_pN/\omega^3\Z_p N: \  x \mapsto \frac{1}{|H|} \sum_{h\in H} h^{-1} \pi(h\cdot x) 
	\end{equation}
	is a projection onto $1-N'+\omega^3 \Z_pN$ as well, but $\widehat \pi\in \End_{\Z_p H}(\omega^2 \Z_pN/\omega^3\Z_p N)$. Therefore the kernel of $\widehat \pi$
	is an $H$-module complement for $1-N'+\omega^3 \Z_pN$.
	The existence of a $\Z_p$-module complement is precisely what is asserted in Theorem \ref{thm_sandling}.
	Since $N$ acts trivially on $\omega^2\Z_p N / \omega^3 \Z_pN $ the complement for $1-N'+\omega^3 \Z_pN$ we just obtained is automatically 
	a $\Z_p N$-module. Hence we have obtained an $H$-stable ideal $I \leq \omega\Z_pN$  such that $I\cap (1-N) = \emptyset$  (here we use the second part of Remark \ref{rem_dim_subgroup}) and 
	\begin{equation}
		|\omega \Z_p N /I| = |\omega \Z_pN / \omega^2\Z_pN|\cdot |\omega^2 \Z_p N / I| = |G/G'| \cdot |G'| = G
	\end{equation}
	where the first part of Remark \ref{rem_dim_subgroup} was used. It follows that $I$ is a complement for the set $1-N \subset \omega \Z_p N$. Therefore the map
	\begin{equation}
		\chi:\ N \longrightarrow \omega \Z_pN/I:\  n \mapsto 1-n +I
	\end{equation}
	is a bijective $1$-cocycle which is $A$-equivariant due to $I$ being stable under the induced action of $A$ on $\Z_p N$.
\end{proof}

The following is another way of constructing an equivariant IYB-struc\-ture on nilpotent groups of class two, provided the order of the group is odd.
\begin{remark}[{see \cite{AultWattersCircleGroups} or \cite[Proposition 9.4]{JOCBraces}}]\label{rem_cedo}
	For a group $N$ of odd order there is a well-known IYB-structure on $N$, by defining the following addition on $N$
	\begin{equation}
		n_1 + n_2 := n_1 n_2 \sqrt{\left [n_2,n_1 \right ]}
	\end{equation}
	and letting $N$ act from the left by the formula $^{n_1}n_2 := n_1n_2+n_1^{-1}$ (owed to the fact that the elements of $N$, when construed as an $N$-module, should correspond to the elements $1-n$ in the corresponding quotient of the augmentation ideal of the group ring).
	It is clear by definition that this IYB-structure is $\Aut(N)$-equivariant. Of course this construction only works when $N$ is of odd order, because only then will square roots of group elements necessarily exist.
\end{remark}

The following  proposition is one of the ingredients required in order to construct an explicit IYB-structure on the counter-example to the isomorphism problem given by Hertweck. It is a slight generalization of \cite[Corollary 3.5]{CFJIYB}, the latter stating that given two groups $G$ and $H$, both being IYB and $H$ being a permutation group, their wreath product 
$G \wr H$ will be IYB as well. 
\begin{prop}\label{prop_wreath}
	Let $G$ be a finite group and let $A$ be a group acting on $G$ by automorphisms. Assume $G$ has an $A$-equivariant IYB-structure. 
	Then $G^{(n)} := G\times \cdots \times G$ ($n$ factors) has an $A \wr \Sigma_n$-equivariant IYB-structure (where $\Sigma_n$ denotes the symmetric group on $n$ points).
\end{prop}
\begin{proof}
	Let $(M,\chi)$ be an $A$-equivariant IYB-structure on $G$. Then $A\wr \Sym_n$ acts on both $G^{(n)}$ and $M^{(n)} := \bigoplus^n M$ in the natural way, by letting
	$(a_1,\ldots, a_n)\cdot \sigma$ (where $\sigma \in \Sym_n$ and $a_1,\ldots,a_n\in A$) send $(x_1,\ldots, x_n)$ (an element of $M^{(n)}$ or $G^{(n)}$) to
	$(a_1\cdot x_{\sigma^{-1}(1)}, \ldots, a_n \cdot x_{\sigma^{-1}(n)})$. We define
	\begin{equation}
		\chi^{(n)}:\ G^{(n)} \longrightarrow M^{(n)}: \  (g_1,\ldots,g_n) \mapsto (\chi(g_1), \ldots, \chi(g_n))
	\end{equation}
	Clearly this is a bijective 1-cocycle, and $A\wr \Sym_n$-equivariance is also clear, since
	\begin{align*}
		\chi^{(n)}({^{(a_1,\ldots, a_n)\cdot \sigma }} (g_1,\ldots,g_n)) &= (\chi(^{a_1}g_{\sigma^{-1}(1)}), \ldots \chi(^{a_n}g_{\sigma^{-1}(n)})) \\
		&= (a_1\cdot \chi(g_{\sigma^{-1}(1)}), \ldots, a_n\cdot \chi(g_{\sigma^{-1}(n)})) \\
		&= {{(a_1,\ldots,a_n)\cdot \sigma}\cdot  \chi^{(n)}(g_1,\ldots, g_n)} \tag*{\qedhere}
	\end{align*}
\end{proof}

It is of course trivial that if $G$ has an $A$-equivariant IYB-structure, and $\varphi: B \longrightarrow A$ is a homomorphism from another group into $A$,
then  $G$ has a $B$-equivariant structure, where $b\in B$ acts on $G$ in the same way as $\varphi(b)$. So it follows more generally that
if a finite group $H$ is IYB and it acts on $G^{(n)}$ through automorphisms induced by elements of $A\wr \Sigma_n$, then $G^{(n)}\rtimes H$ is IYB. We can also 
write down its IYB-structure explicitly provided we know (explicitly) a $B$-equivariant IYB-structure on $G$ and an IYB-structure on $H$.

\section{Hertweck's counterexample to the isomorphism problem}

In \cite{HertweckIso}, M. Hertweck gave an example of a finite group $X$ such that $\Z X \iso \Z Y$ for some other finite group $Y$ not isomorphic to $X$. On the other hand it is well-known that the circle group of a radical ring (a concept closely related to IYB-groups) has the property that $\Z X \iso \Z Y$ implies $X \iso Y$ (see for instance \cite[Chaper 9.4]{MiliesSehgal}). Therefore there was some hope that Hertweck's group $X$ might be an example of a solvable group which is not IYB. However this is not the case. Hertweck's counterexample $X$ is a semidirect product $Q \rtimes P$, where $Q$ is a $97$-group of nilpotency class two, and $P$ is a $2$-group which is a semidirect product of two abelian groups. By Theorem \ref{thm_semidirect_nilpot} it is already clear that $X$ is indeed IYB.
Note that the group $Y$ with $\Z Y \iso \Z X$ is constructed in the same way as $X$, implying that it is IYB as well. In this section we are going to explain how to explicitly construct an IYB-structure on $X$. There clearly is no problem constructing an IYB-structure on $P$. The group $Q$ is a direct product of four cyclic groups of order $97$, and eight times a group $D$ which is described below. $P$ acts separately on the direct factors involving cyclic groups and those involving $D$, acting through $\Aut(D) \wr \Sym_8$ on the latter. Due to Proposition 
\ref{prop_wreath} it hence suffices to find an IYB-structure on $D$ which is equivariant with respect to the right subgroup of $\Aut(D)$. 

Now let $q$ be a prime such that $q \equiv 1 \mod 4$, and let
\begin{equation}
	D := \left( (d_3:d_3^q) \times (d_2 :d_2^q) \right) \rtimes (d_1:d_1^q)
\end{equation}
where $d_3$ is central and $d_2^{d_1} = d_2 d_3$ (note that this is isomorphic to the $q$-Sylow subgroup of $\GL_2(q)$). Define 
the following automorphisms of $D$:
\begin{equation}
	\tau: \left\{\begin{array}{c} d_1 \mapsto d_2\\ d_2 \mapsto d_1 \\ d_3 \mapsto d_3^{-1} \end{array}\right.\quad
	\alpha_1: \left\{\begin{array}{c} d_1 \mapsto d_1^{\zeta}\\ d_2 \mapsto d_2 \\ d_3 \mapsto d_3^{\zeta} \end{array}\right.\quad 
	\alpha_2: \left\{\begin{array}{c} d_1 \mapsto d_1\\ d_2 \mapsto d_2^{\zeta} \\ d_3 \mapsto d_3^{\zeta} \end{array}\right.
\end{equation}
where $\zeta$ is a generator of the multiplicative group $(\Z/q)^\times$. Note that $\tau^2=\id$, $\tau \notin \langle \alpha_1, \alpha_2 \rangle = \langle \alpha_1\rangle \times \langle \alpha_2 \rangle$ (since the $\alpha_i$ stabilize the subgroup $\langle d_1 \rangle$, but $\tau$ does not). Moreover $\alpha_1^\tau = \alpha_2$ and $\alpha_2^\tau = \alpha_1$. Hence it follows that $A := \langle \tau, \alpha_1, \alpha_2 \rangle$ is isomorphic to $(C_{q-1} \times C_{q-1})\rtimes C_2$. In particular it has order $2\cdot (q-1)^2$. It has a faithful representation over $\F_q^2$, which can be obtained by considering the action of $A$ on
$Q/Z(Q)$:
\begin{equation}
	\resizebox{.9\hsize}{!}{$
	\Delta:\ A \longrightarrow \GL_2(q):\ \tau \mapsto \matZ{0&1\\1&0} \quad \alpha_1 \mapsto \matZ{\zeta&0\\0&1} \quad \alpha_2 \mapsto \matZ{1&0\\0&\zeta}
	$}
\end{equation}
Note moreover that since $q \equiv 1 \mod 4$ the group $A$ contains a $2$-Sylow subgroup of $\Aut(Q)$. To see this note that $\GL_2(q)$ has order 
$(q^2-1)\cdot (q^2-q) = q\cdot 2\cdot (q-1)^2 \cdot \frac{q+1}{2}$, and $\frac{q+1}{2}$ is odd. Therefore $\Delta(A)$ contains a $2$-Sylow subgroup of 
$\GL_2(q)$. The elements in $\Aut(Q)$ which act trivial on $Q/Z(Q)$ must send $d_1$ to $d_1d_3^{i_1}$, $d_2$ to $d_2d_3^{i_2}$ and $d_3 = d_2^{d_1}d_{2}^{-1}$ to $d_3$ for some $i_1, i_2 \in \Z$. It is easily seen that the order of such an automorphism is either $q$ or one. Hence the $q'$-Sylow subgroups of $\Aut(Q)$ have the same order as those of $\GL_{2}(q)$.

Now we would like to describe an $A$-equivariant IYB-structure on $D$. Together with \ref{prop_wreath} and the known construction of an IYB-structure on a semidirect product of abelian groups this is enough to piece together an IYB-structure on the group given by Hertweck. We will indeed see that, up to isomorphism, $D$ has a unique $A$-equivariant IYB structure.
While this uniqueness is not necessary as far as constructing an IYB-structure on Hertweck's counter-example goes, it shows that there is not much room for different
equivariant IYB-structures (whereas computational experiments seem to indicate that non-equivariant IYB-structures on $p$-groups exist in abundance).
\begin{thm}\label{thm_upper_triang}
	Let $M$ be the vector space $\F_q^3$ with the following left action of $D\rtimes A$:
	\begin{equation}
		\resizebox{.9\hsize}{!}{$
		d_1 \mapsto \matD{1&1&0\\0&1&0\\0&0&1} \quad d_2 \mapsto \matD{1&0&1\\0&1&0\\0&0&1} \quad d_3 \mapsto \id  \quad \underbrace{a}_{\in A} \mapsto \det(\Delta(a))\cdot \left(\begin{array}{cc}1 & \begin{array}{cc}0&0\end{array} \\ \begin{array}{c}0\\0\end{array} & \Delta(a^{-1})^\top\end{array}\right)$}
	\end{equation}
	Then the map 
	\begin{equation}\label{eqn_sdjkdjlkjd}
		\chi: D \longrightarrow M: \  d_1^{n_1}d_2^{n_2}d_3^{n_3} \mapsto \vect{n_3 - \frac{1}{2}\cdot  n_1n_2\\ -\frac{1}{2}\cdot n_2 \\ \frac{1}{2} \cdot n_1}
	\end{equation}
	defines an $A$-equivariant IYB-structure on $D$, and $(M,\chi)$ is up to isomorphism the only $A$-equivariant IYB-structure on $D$.
\end{thm}
\begin{proof}
	We start with an arbitrary $A$-equivariant IYB-structure $(M,\chi)$ on $D$ and show that it has to be of the claimed form.
	First note that $A$ contains the element $\tau \alpha_1$ of order $2(q-1)$. The $q$-adic group ring $\Z_q C_{2(q-1)}$ decomposes as a direct sum
	of the unramified extension $\Z_q[\zeta_{2(q-1)}]$ and a number of copies of $\Z_q$. The latter correspond to non-faithful representations. 
	Hence any faithful $\Z_q C_{2(q-1)}$-module has a direct summand of the form $\Z_q[\zeta_{2(q-1)}]/q^i$ with $i\geq 1$. Since $A$ acts faithfully on $D$ it must act in such a manner on $M$ as well, forcing $M$ to be the direct sum $\Z_q[\zeta_{2(q-1)}]/q \oplus \Z_q/q$ (as all other possibilities would have a cardinality that is too large). That is, $M$ has to be an $\F_q$-vector space.
	
	Next note that the kernel of the action of $D$ on $M$ must be a normal subgroup of $D$ which is stable under the action of $A$. There are only three such
	subgroups of $D$, namely $1$,  $Z(D)=\langle d_3\rangle$ and $D$ itself. Clearly $D$ must act non-trivially on $M$, since otherwise  $\chi$
	would be an isomorphism between $D$ and the additive group of $M$. If $D$ were to act faithfully on $M$, then it would map onto a $q$-Sylow subgroup
	of $\Aut((M,+))\iso \GL_3(q)$. Then $A$ would be embedded in the normalizer of such a $q$-Sylow subgroup. This is impossible, since such a normalizer is isomorphic to the group of invertible upper triangular $3\times 3$-matrices over $\F_q$, and there is no element of order $2(q-1)$ in that group.
	
	We have hence established that $M$ is a faithful $D/Z(D)$-module. Clearly $M$ cannot be semisimple (because then $D$ would act trivially). Assume \linebreak
	$\Rad^2(M)\neq 0$. Then $M > \Rad(M) > \Rad^2(M) > 0$ would be a flag which must be stabilized by $A$, again forcing $A$ into a subgroup of $\GL_3(q)$ isomorphic to the group of invertible upper triangular matrices, which was impossible. Hence $\Rad^2(M)=0$ and $\Rad(M)$ is either one or two dimensional and semisimple. 
	Assume $\Rad(M)$ is two-dimensional. Then $\Rad(M)=\Soc(M)$ (since otherwise $M$ would be semisimple). But for any $z \in Z(D)$ and any $g\in D$ we have
	$g\chi(z)+\chi(g) =\chi(gz)=\chi(zg)= z\chi(g)+\chi(z)=\chi(g)+\chi(z)$, which implies $\chi(z) \in \Soc(M)$. But then $\chi(z^i)=i\chi(z)$ for all $i\in \Z$, which implies that $\chi(Z(D))$ is a one-dimensional subspace of $\Rad(M)$.  But $A$ must fix this subspace (since $A$ fixes $Z(D)$), and therefore
	$A$ fixes the flag $M > \Rad(M) > \chi(Z(D)) > 0$, which is impossible by the same arguments as before. It follows hence that $\Rad(M)$ is one-dimensional.
	The dual of $M$ is thus an $\F_q D$-module with simple top and radical length two. This means that this dual is an epimorphic image of $\F_q D / \Rad^2(\F_qD)$, and it is in fact isomorphic since the latter also has dimension $3$. Hence $M$ is isomorphic (as an $\F_q D$-module) to the dual of $\F_q D / \Rad^2(\F_qD)$.
	So we can assume without loss that $d_1$ and $d_2$ act on $M$ as claimed.
	Write $\Delta_M:\ D\rtimes A \longrightarrow \GL_3(q)$ for the action of $D\rtimes A$ on $M$.
	Then it follows, using the identity $\Delta_M(a) \cdot \Delta_M(d_i)\cdot \Delta_M(a^{-1}) = \Delta_M({^a d_i})$ for all $a\in A$ and $i\in\{1,2\}$, that an element $a\in A$ acts in the following way
	\begin{equation}
		\left( \begin{array}{c|c}
			\varphi(a) & \begin{array}{cc}g_1(a) & g_2(a) \end{array}\\\hline 
			\begin{array}{c} 0 \\ 0 \end{array} & \varphi(a)\cdot \Delta(a^{-1})^{\top}
		\end{array} \right)
	\end{equation}
	where $\varphi:\ A \longrightarrow \F_q^\times$ is a group homomorphism and $g_1, g_2$ are arbitrary functions from $A$ to $\F_q$. We can conjugate the matrices in such a way that $g_1$ and $g_2$ both become the zero map (just find an $A$-stable complement for the subspace generated by the first standard basis vector).
	
	Since we already know that $Z(D)$ maps into the socle of $M$ we can say that $\chi(d_3)$ must be a non-zero multiple of the first standard basis vector.
	By applying an automorphism of $M$ we can hence choose $\chi(d_3)$ to be the first standard basis vector, and therefore $\chi(d_3^i)=i\cdot \chi (d_3)$ for all $i\in \Z$. The $A$-equivariance of $\chi$ now implies $\varphi = \det \circ \Delta$. Now we determine the images $\chi(d_1)$ and $\chi(d_2)$. The fact that $\alpha_2$ fixes $d_1$ implies that $\chi(d_1)$ lies in the eigenspace of $\Delta_M(\alpha_1)$ with associated eigenvalue one. So $\chi(d_1)=(0,0,x)^\top$ for some $x \in \F_q^\times$. In the same manner one sees that $\chi(d_2)=(0,y,0)^\top$ for some $y \in \F_q^\times$. Using $^\tau d_1 = d_2$ we can conclude $x=-y$ (again by $A$-equivariance). By comparing $\chi(d_1d_2) = d_1\chi(d_2)+\chi(d_1)=(-x,-x,x)^\top$ and $\chi(d_1d_2)=\chi(d_3^{-1}d_2d_1)=
	-\chi(d_3)+d_2\chi(d_1)+\chi(d_2)=(x-1,-x,x)^\top$ we infer that $-x=x-1$, i. e. $x=-1/2$. Formula (\ref{eqn_sdjkdjlkjd}) now follows easily:
	$
		\chi(d_1^{n_1}d_2^{n_2}d_3^{n_3}) = n_3\chi(d_3) + d_1^{n_1} {n_2}\cdot  \chi(d_2)+n_1\cdot \chi(d_1)
	$.
\end{proof}

\section{Concluding remarks}

There are actually not that many groups of small order which cannot be seen to be IYB using Theorem \ref{thm_semidirect_nilpot}. Using the \textit{SmallGroups} library that comes with the computer algebra system \textsc{Gap} (\cite{GAP4}) one can show that the potential counterexamples of order
$\leq 200$ have orders $48$, $96$, $144$, $162$ and $192$ (excluding prime powers). For all except the groups of order $192$ one can compute an IYB-structure using a brute-force approach (a slightly more intelligent approach might actually work for the groups of order 192 as well). For (small) $p$-groups one can use a heuristic approach to compute an IYB-structure. 
Namely, given a $p$-group $P$ one can pick a normal subgroup $N=\langle n \rangle$ of order $p$. Assume we already computed an ideal $I\leq \omega \Z_p G/N$ complementing $1-G/N$. Then we can take the preimage of $I$ under the natural epimorphism $\omega \Z_p G \twoheadrightarrow \omega \Z_p G/N$ and try to find a maximal submodule in this preimage which does not contain $1-n$. If such a maximal submodule exists then it is indeed an ideal in $\omega \Z_p G$ complementing $1-G$. While such a maximal submodule does not always exist, trying this multiple times with different subgroups $N$ and different choices of $I$ will typically work (actually no
potential counterexamples were found in this way). Using an implementation of this heuristic in \textsc{Gap} it was possible to show that all $2$-groups of order up to (and including) $512$ are IYB. Also all other $p$-groups of order strictly less than $1024$ turned out to be IYB. So the evidence that all nilpotent or even all solvable are IYB seems to be piling up.

\paragraph*{Acknowledgments} This research was supported by the \textit{Fonds Wetenschappelijk Onderzoek - Vlaanderen (FWO)  project G.0157.12}. 
I would also like to thank F. Ced\'{o} for looking at a preliminary version of this paper and suggesting Remark \ref{rem_cedo}.
%
\bibliographystyle{alpha}
\bibliography{refs}

\end{document}